\documentclass[fleqn,11pt]{article}
\usepackage{epic}
\usepackage{epsfig}
\usepackage[english]{babel}
\usepackage[mathscr]{eucal}
\usepackage{todonotes}

\usepackage{amssymb,epsf,amsmath}
\usepackage[T1]{fontenc}

\newtheorem{Theorem}{Theorem}[section]
\newtheorem{Proposition}[Theorem]{Proposition}
\newtheorem{Definition}[Theorem]{Definition}
\newtheorem{Lemma}[Theorem]{Lemma}

\newtheorem{Corollary}[Theorem]{Corollary}

\newtheorem{Remark}[Theorem]{Remark}

\newtheorem{Ex}{Example}[section]

\newcommand{\Halmos}{\hfill$\Box$}
\newenvironment{Example}{\begin{Ex}\rm}{\smallskip\end{Ex}}

\newtheorem{proofhead}{Proof.}

\newenvironment{proof}
    {
        \par
        \begin{proofhead}
        \normalfont
    }
    {   \qed
        \end{proofhead}
        \par
    }

\newcommand{\bb}{\mathbb}

\def \pr{{\bb P}}

\newcommand{\qed}{\hfill$\Box$}

\newcommand{\refs}[1]{(\ref{#1})}

\newcounter{mylistcnt}
\renewcommand{\themylistcnt}{{\rm({\roman{mylistcnt}})}}

\newcounter{zad}

\newcommand{\toi}{\to\infty}

\newenvironment{namelist}[1]{%
\begin{list}{}
     {
      
      \settowidth{\labelwidth}{#1}
      \setlength{\leftmargin}{1.1\labelwidth}
               }
      }{%
\end{list}}

\begin{document}

\title{Asymptotic properties of extremal Markov processes driven by Kendall convolution}%
\author{
Marek  Arendarczyk$^1$, Barbara Jasiulis - Go\l dyn$^2$, Edward Omey$^3$
\vspace*{2ex} \\
$^{1,2}$Mathematical Institute, University of Wroc\l aw, \\ pl. Grunwaldzki 2/4, 50-384 Wroc\l aw, Poland.\\
$^{3}$Faculty of Economics and Business-Campus Brussels, \\ KU Leuven, Warmoesberg 26, 1000 Brussels, Belgium
\vspace*{2ex} \\
E-mail: $^1$marendar@math.uni.wroc.pl, \\
$^2$jasiulis@math.uni.wroc.pl, \\
$^3$edward.omey@kuleuven.be
}

\vspace{0.6cm}
\maketitle
\tableofcontents
\newpage
\begin{abstract}
This paper is devoted to the analysis of the finite-dimensional distributions and asymptotic behavior of extremal Markov processes connected to the Kendall convolution.
In particular, based on its stochastic representation, we provide general formula for finite dimensional distributions of the random walk driven by the Kendall convolution for a large class of step size distributions.
Moreover, we prove limit theorems for random walks and connected continuous time stochastic process.

\end{abstract}

\section{Introduction}

The Kendall convolution being the main building block in the construction of the extremal Markov process called Kendall random walk $\{X_n \colon n \in \mathbb{N}_0 \}$
is an important example of generalization of the convolutions corresponding to classical sum and to classical maximum.
Originated by Urbanik \cite{Urbanik64} (see also \cite{King}) generalized convolutions regain popularity in recent years (see, e.g., \cite{BJMR, renewalKendall, Misiewicz2019} and references therein).
In this paper we focus on the Kendall convolution (see, e.g., \cite{uniqKendall, Misiewicz2018}) which thanks to its connections to heavy tailed distributions,
Williamson transform \cite{factor}, Archimedian copulas, renewal theory \cite{renewalKendall}, non-comutative probability \cite{JasKula} or delphic semi-groups \cite{Gilewski1, Gilewski2, Kendall} presents high potential of applicability.
We refer to \cite{BJMR} for the definition and detailed description of the basics of theory of generalized convolutions, for a survey on the most important classes of generalized
convolutions, and a discussion on L\'evy processes and stochastic integrals based on that convolutions, as well as to \cite{KendallWalk} for the definition and
basic properties of Kendall random walks.

Our main goal is to study finite dimensional distributions and asymptotic properties of extremal Markov processes
connected to the Kendall convolution.
In particular we present many examples of finite dimensional distributions by the unit step characteristics, which create a new class of heavy tailed distributions with potential applications.
Innovative thinking about possible applications comes from the fact that the generalized convolution of two point-mass probability measures can be a non-degenerate probability measure.
In particular the Kendall convolution of two probability measures concentrated at $1$ is the Pareto distribution and consequently it generates heavy tailed distributions.
In this context the theory of regularly varying functions (see, e.g., \cite{Bin71, Bin84}) plays a crucial role. In this paper, regular variation techniques are used, to investigate asymptotic behavior of Kendall random walks and convergence of the finite dimensional distributions of continuous time stochastic processes constructed from Kendall random walks.

Most of the proofs presented in this paper are also based on the application of Williamson transform  that is a generalized characteristic function of the probability measure in the Kendall algebra (see, e.g., \cite{BJMR, Williamson}). The great advantage of the Williamson transform is that it is easy to invert. It is also worth to mention that,
this kind of transforms is generator of Archimedean copulas \cite{Neslehova1,Neslehova2}
and is used  to compute radial measures for reciprocal Archimedean copulas \cite{GenNesRiv}.
Asymptotic properties of the Williamson transform in the context of Archimedean copulas and extreme value theory were given in \cite{Larsson}.
In this context, we believe, that results on Williamson transform obtained in presented paper might be applicable in copula theory
which can be an interesting topic for future research.
\\

We start, in Section \ref{sec:def} with presenting the definition and basic properties of the Kendall convolution.
Next, the definition and construction of random walk under the Kendall convolution is presented, which leads to a stochastic representation of the process $\{X_n\}$. The basic properties of the transition probability kernel are also proved.
We refer to \cite{KendallWalk, Poisson, renewalKendall, factor} for discussions and proofs of further properties of the process $\{X_n\}$.
The structure of the processes considered here (see Definition \ref{def_krm}) is similar to the first order autoregressive maximal Pareto processes \cite{Arnold1, Arnold2, Lopez, Yeh}, max-AR(1) sequences \cite{Alpuim2}, minification processes \cite{Lewis, Lopez}, the max-autoregressive moving average processes MARMA \cite{Ferreira}, pARMAX and pRARMAX processes \cite{pRARMAX}.
Since the random walks form a class of extremal Markov chains, we believe that studying them will yield an important contribution to extreme value theory. Additionally, the sequences in the presented paper have interesting dependency relationships between factors which also justifies their potential applicability.

Section \ref{sec:fdd} is devoted to the analysis of the finite-dimentional distributions of the process $\{X_n\}$.
We derive general formula and present some important examples of processes with different types of step distributions which leads to new classes of heavy tailed distributions.

In Section \ref{sec:lim} we study different limiting behaviors of Kendall convolutions and connected processes. We present asymptotic properties of random walks under Kendall convolution using regular variation techniques.
In particular we show asymptotic equivalence between Kendall convolution and maximum convolution in the case of regularly varying step distribution $\nu$. The result shows the connection with classical extreme value theory
(see, e.g., \cite{Embr97}) and suggests
possible applications of the Kendall random walk in modelling phenomenons where independence between events can not be assumed.
Moreover limit theorems for Kendall random walks are given in the case of finite $\alpha$-moment as well as in the case of regularly varying tail of unit step distribution. Finally, we define continuous time stochastic process based on random walk under Kendall convolution and prove convergence of its finite-dimensional distributions using regular variation techniques. \\

{\it Notation.}

Through this paper, the distribution of the random
element $X$ is denoted by $\mathcal{L}(X)$.
By $\mathcal{P}_+$ we denote family of all probability measures on the Borel $\sigma$-algebra $\mathcal{B}(\mathbb{R}_+)$
 with $\mathcal{\mathbb{R}_+} := [0, \infty)$.
For a probability measure $\lambda \in \mathcal{P}_{+}$ and $a \in
\mathbb{R}_+$ the rescaling operator is given by $\mathbf{T}_a \lambda = \mathcal{L}(aX)$ if $\lambda = \mathcal{L}(X)$.
The set of all natural numbers including zero is denoted by $\mathbb{N}_0$. Additionally we use notation $\pi_{2\alpha}, \alpha > 0$ for a Pareto random measure with probability density function (pdf)
\begin{equation} \label{pareto:pdf}
    \pi_{2\alpha}\, (dy)
=
    2\alpha y^{-2\alpha-1} \pmb{1}_{[1,\infty)}(y)\, dy.
\end{equation}
Moreover, by
\[
		m_{\nu}^{(\alpha)}
	:=
		\int_0^\infty x^\alpha \nu(dx)
\]		
we denote the $\alpha$th moment of measure $\nu$. The truncated $\alpha$-moment of the measure $\nu$
with cumulative distribution function (cdf) $F$ is given by	
\[
H(t)
	:=
\int\limits_0^t y^{\alpha} F (dy).
\]

By $\nu_1 \vartriangle_{\alpha} \nu_2$ we denote the Kendall convolution of the measures $\nu_1, \nu_2 \in \mathcal{P}_{+}$.
For all $n \in \mathbb{N}$,  the Kendall convolution of $n$ identical measures $\nu$ is denoted by
$\nu^{ \vartriangle_{\alpha} n} := \nu \vartriangle_{\alpha} \dots \vartriangle_{\alpha} \nu$ (n-times).
By $F_n$ we denote the cumulative distribution function of the measure $\nu^{\vartriangle_{\alpha}n}$,
whereas for the tail distribution of $\nu^{\vartriangle_{\alpha}n}$ we use the standard notation $\overline{F}_n$.
By $\stackrel{d}{\to}$ we denote convergence in distribution, whereas
$\stackrel{fdd}{\to}$ denotes convergence of finite-dimentional distributions.
Finally, we say that a measurable and positive function $f(\cdot)$ is regularly varying at infinity
with index $\beta$ if, for all $x>0$, it satisfies $\lim_{t\rightarrow \infty }\frac{f(tx)}{f(t)}=x^{\beta }$
(see, e.g., \cite{Bin87}).

\section{Stochastic representation and basic properties of Kendall random walk} \label{sec:def}

We start with the definition of the Kendall generalized convolution (see, e.g., \cite{BJMR}, Section 2).
\begin{Definition} \label{Kendal:def}
The binary operation $\vartriangle_{\alpha} \colon \mathcal{P}_+^2 \rightarrow \mathcal{P}_+$ defined for point-mass measures by
\begin{equation*} 
\delta_x \vartriangle_{\alpha} \delta_y := T_M \left( \varrho^{\alpha} \pi_{2\alpha} + (1- \varrho^{\alpha}) \delta_1 \right),
\end{equation*}
where $\alpha > 0$, $\varrho = {m/M}$, $m = \min\{x, y\}$, $M = \max\{x, y\}$, is called the Kendall convolution.
The extension of $\vartriangle_{\alpha}$ for any $\mathcal{B}(\mathbb{R}_+)$ and  $\nu_1, \nu_2 \in \mathcal{P}_+$ is given by
\begin{equation} \label{conv}
\nu_1 \vartriangle_{\alpha} \nu_2 (A)  = \int\limits_{0}^{\infty} \int\limits_{0}^{\infty} \left(\delta_x \vartriangle_{\alpha} \delta_y\right) (A)\, \nu_1(dx) \nu_2(dy).
\end{equation}

\end{Definition}
\begin{Remark}
Note that the convolution of two point mass measures is a continuous measure that reduces to a Pareto distribution $\pi_{2\alpha}$ in case of $x = y = 1$, which is different than the classical convolution or maximum convolution algebra, where convolution of discrete measures yields also a discrete one.
\end{Remark}

In the Kendall convolution algebra the main tool used in the analysis of a measure $\nu$ is the Williamson transform
(see, e.g., \cite{Williamson}) that is characteristic function for Kendall convolution (see, e.g., \cite{BJMR}, Definition 2.2) and plays the same role as the classical Laplace or Fourier transform for convolutions defined by addition of independent random elements.
We refer to \cite{BJMR}, \cite{Misiewicz2018} and \cite{Urbanik64}, for the definition and detailed discussion on properties of generalized characteristic functions and its connections to generalized convolutions. Through this paper, the function $G(t)$ that is Williamson transform at point $\frac{1}{t}$
plays a crucial role in the analysis of Kendall convolutions and connected stochastic processes.

\begin{Definition} \label{Williamson:def}
The operation 
$G: \mathbb{R_+} \rightarrow \mathbb{R_+}$ given by
\begin{equation*} 
G(t) = \int\limits_{0}^{\infty} \Psi\left(\frac{x}{t}\right) \nu(dx), \quad \nu \in \mathcal{P}_+,
\end{equation*}
where
\begin{equation}\label{Psi}
\Psi\left(\frac{x}{t}\right) = \left( 1 - \left(\frac{x}{t}\right)^{\alpha} \right)_+,
\end{equation}
$a_+ = \max(0,a)$, $\alpha > 0$, is called the Williamson transform of measure $\nu$ at point $\frac{1}{t}$.
\end{Definition}

\begin{Remark}
Note that
$\Psi\left(\frac{x}{t}\right)$
as a function of $t$ is the Williamson transform of $\delta_x, x\ge 0$.
\end{Remark}

\begin{Remark}	\label{will:prod}
Due to Proposition 2.3 and Example 3.4 in \cite{BJMR} function $G\left(\frac{1}{t}\right)$ is a generalized characteristic function for the Kendall convolution.
Thus, the Williamson transform $G_n(t)$
of the measure $\nu^{ \vartriangle_{\alpha} n}$ has the following, important, property (see, e.g., Definition 2.2 in \cite{BJMR})
\begin{equation} \label{Williamson:n}
G_n(t)=G(t)^n.
\end{equation}
\end{Remark}

By using recurrence construction, we define a stochastic processes $\{X_n: n \in \mathbb{N}_0\}$, called the Kendall random walk
(see also \cite{KendallWalk, factor, renewalKendall}). Further, we show strict connection of the process $\{X_n\}$ to the Kendall convolution.

\begin{Definition} \label{def_krm} The stochastic process $\{X_n \colon n \in \mathbb{N}_0\}$ is a discrete time Kendall random walk  with parameter $\alpha>0$ and step distribution $\nu \in \mathcal{P}_+$ if there exist
\begin{namelist}{ll}
\item[\bf 1.] $\{Y_k\}$ i.i.d. random variables with distribution $\nu$,
\item[\bf 2.] $\{\xi_k\}$ i.i.d. random variables with uniform distribution on $[0,1]$,
\item[\bf 3.] $\{\theta_k\}$ i.i.d. random variables with Pareto distribution with $\alpha > 0$ and density
		\[
			\pi_{2\alpha}\, (dy)
				=
			2\alpha y^{-2\alpha-1} \pmb{1}_{[1,\infty)}(y)\, dy,
		\]
\end{namelist}
such that sequences $\{Y_k\}$, $\{\xi_k\}$, and $\{\theta_k\}$ are independent and
\[
X_0 = 1, \quad  X_1 = Y_1, \quad X_{n+1} = M_{n+1}\left[ \pmb{1}_{(\xi_n > \varrho_{n+1})} + \theta_{n+1} \pmb{1}_{(\xi_n < \varrho_{n+1})}\right],
\]
where 
\[
M_{n+1} = \max\{ X_n, Y_{n+1}\}, \quad m_{n+1} = \min\{ X_n, Y_{n+1}\}, \quad \varrho_{n+1} = \frac{m_{n+1}^{\alpha}}{M_{n+1}^{\alpha}}.
\]
\end{Definition}
In the next proposition we show that the process constructed in Definition \ref{def_krm} is a homogeneous Markov process driven by the Kendall convolution.
We refer to \cite{BJMR}, Section 4 for the proof and a general discussion of the existence of the Markov processes under generalized convolutions.
\begin{Proposition} \label{kernel:prop}
The process $\{X_n : n \in \mathbb{N}_0\}$ with the stochastic representation given by Definition \ref{def_krm} is a
homogeneous Markov process with transition probability kernel
\begin{equation} \label{kernel}
P_{n,n+k}(x, A)
:=
P_k(x, A)
=
\left( \delta_x \vartriangle_\alpha \nu^{\vartriangle_\alpha k} \right) (A),
\end{equation}
where $k, n \in \mathbb{N}, A \in
\mathcal{B}(\mathbb{R}_{+}), x\ge 0, \alpha >0$.
\end{Proposition}
The proof of Proposition \ref{kernel:prop} is presented in Section \ref{kernel:prop:proof}.

\section{Finite dimensional distributions} \label{sec:fdd}
In this section we study the finite dimensional distributions of the process $\{X_n\}$.
We start with a proposition
that describes the one-dimensional distributions of
$\{X_n\}$ and their relationships with the Williamson transform and truncated $\alpha$-moment.

\begin{Proposition}\label{prop:Fn}
Let $\{X_n : n \in \mathbb{N}_0\}$ be a Kendall random walk with parameter $\alpha>0$
and unit step distribution $\nu \in\mathcal{P}_+$ with cdf $F$.
Then
\begin{itemize}
\item[{\rm (i)}]
for any $t \ge 0$ we have
\[
G(t) = \alpha t^{-\alpha}  \int\limits_0^t x^{\alpha-1} F(x) dx,
\]
and
$F(t) =  G(t)  + \frac{t}{\alpha} G'(t)$.
\item[{\rm (ii)}]
    for any $t \ge 0, n \ge 1$ we have
\begin{eqnarray*}
F_n(t) =  G(t)^{n-1} \bigl[n t^{-\alpha} H(t) + G(t) \bigr].
\end{eqnarray*}
\end{itemize}
\end{Proposition}
\begin{proof}
First, observe that
\begin{equation}	\label{F:eq}
		G(t)
	=
		F(t) - t^{-\alpha} \int\limits_0^t x^{\alpha} \nu (dx)
	=
		\alpha t^{-\alpha}  \int\limits_0^t x^{\alpha-1} F(x) dx,
\end{equation}
where the last equation follows from integration by parts. In order to complete the proof of (i) it suffices to
take derivatives on both sides of the above equation.

In an analogous way we obtain that
\begin{equation} \label{Hn:eq}
		(G(t))^n = \alpha t^{-\alpha} \int_0^t x^{\alpha - 1} F_n(x) dx.
\end{equation}
In order to complete the proof it is sufficient to take derivatives on both sides of equation \refs{Hn:eq} and apply (i).
\end{proof}

The next two lemmas give characterizations of the transition probabilities of the process $\{X_n\}$ and play an important role in the analysis
 of its finite-dimensional distributions.
\begin{Lemma}\label{lem:transit}
Let $\{X_n : n \in \mathbb{N}_0\}$ be a Kendall random walk with parameter $\alpha>0$ and unit step distribution $\nu \in\mathcal{P}_+$.
For all $k, n \in \mathbb{N}, x, y, t \ge 0$ we have
\begin{itemize}
\item[{\rm (i)}]
\begin{eqnarray} \label{kernel:W}
P_n (x, ((0,t]))
& = &  \nonumber
	\label{int_conv1} \left( G(t)^n + \frac{n}{t^{\alpha}} H(t) G(t)^{n-1} \Psi\left(\frac{x}{t}\right)\right) \mathbf{1}_{\{x < t \}}
\end{eqnarray}
\item[{\rm (ii)}]
\begin{eqnarray} \label{inte1}
     \int\limits_0^t w^{\alpha} P_n(x,dw)
  = \nonumber
      \left( x^{\alpha} G(t)^n + nG(t)^{n-1} H(t) \Psi\left(\frac{x}{t}\right) \right) \mathbf{1}_{\{x < t \}}.
\end{eqnarray}
\end{itemize}
\end{Lemma}

The proof of Lemma \ref{lem:transit} is presented in Section \ref{lem:transit:proof}.	\\

The following lemma is the main tool in finding the finite-dimensional distributions of $\{X_n\}$.
In order to formulate the result it is convenient to introduce the notation
\[
    \mathcal{A}_k := \{0,1\}^k \backslash \{(0,0,...,0) \},\ \ \ \ \
    {\rm for \ any}\  k\in \mathbb{N}.
\]
Additionally, for any $(\epsilon_1,...,\epsilon_n) \in \mathcal{A}_n$
we denote
\[
\tilde{\epsilon}_1 = \min \left\{i \in \{1,..., k\} : \epsilon_i = 1 \right\},  \; \dots ,
\tilde{\epsilon}_m := \min \left\{i>\tilde{\epsilon}_{m-1} : \epsilon_i = 1 \right\}, \ \ \ m=1,2, \dots , s, \ \ \ s := \sum_{i = 1}^k \epsilon_i.
\]

\begin{Lemma}\label{thm:a}
Let $\{X_n : n \in \mathbb{N}_0\}$ be a Kendall random walk with parameter $\alpha>0$ and unit step distribution $\nu \in\mathcal{P}_+$.
Then for any $0 = n_0 \le n_1\le n_2\cdots \le n_k$,
where $n_j \in \mathbb{N}$ for all $j \in \mathbb{N}$ and $0 \le y_0 \le x_1 \le x_2 \le \cdots \le x_k \le x_{k+1}$ we have
\begin{eqnarray*}
\lefteqn{
		\int\limits_{0}^{x_1}\int\limits_{0}^{x_2} \cdots \int\limits_{0}^{x_k}
		\Psi\left(\frac{y_k}{x_{k+1}}\right) P_{n_k-n_{k-1}}(y_{k-1}, dy_k) P_{n_{k-1}-n_{k-2}}(y_{k-2}, dy_{k-1})
		\cdots P_{n_1}(y_0, dy_1) }\\
	&=&
		\sum\limits_{({\epsilon}_1,{\epsilon}_2,\cdots,{\epsilon}_k) \in \mathcal{A}_k}
		\Psi \left( \frac{y_0}{x_{\tilde{\epsilon}_1}} \right)\Psi \left( \frac{x_{\tilde{\epsilon}_s}}{x_{k+1}} \right)
		\prod_{i=1}^{s-1} \Psi\left(\frac{x_{\tilde{\epsilon}_i}}{x_{\tilde{\epsilon}_{i+1}}}\right)
		\prod_{j=1}^k (G(x_j))^{n_j - n_{j-1} - \epsilon_j} \left(\frac{(n_j - n_{j-1})H(x_j)}{x_j^\alpha}\right)^{\epsilon_j} \\
  &+& \Psi \left( \frac{y_0}{x_{k+1}} \right) \prod_{j=1}^k (G(x_j))^{n_j - n_{j-1}},
\end{eqnarray*}
where
\[
		\prod_{i=1}^{s-1} \Psi\left(\frac{x_{\tilde{\epsilon}_i}}{x_{\tilde{\epsilon}_{i+1}}}\right) = 1 \ \ \ {\rm for} \ \ \ s = 1.
\]
\end{Lemma}
The proof of Lemma \ref{thm:a} is presented in Section \ref{thm:a:proof}.\\

Now, we are able to derive a general formula for the finite-dimentional distributions of the process $\{X_n\}$.
\begin{Theorem}\label{thm:b}
Let $\{X_n : n \in \mathbb{N}_0\}$ be a Kendall random walk with parameter $\alpha>0$ and unit step distribution $\nu \in\mathcal{P}_+$. Then for any $0 =: n_0 \le n_1\le n_2\cdots \le n_k$,
where $n_j \in \mathbb{N}$ for all $j \in \mathbb{N}$ and $0 \le x_1 \le x_2 \le \cdots \le x_k$ we have
\begin{eqnarray*}
		\lefteqn{ \pr(X_{n_k} \le x_k, X_{n_{k-1}} \le x_{k-1}, \cdots, X_{n_1} \le x_1)}\\
	&=&
		\sum\limits_{({\epsilon}_1,{\epsilon}_2,\cdots,{\epsilon}_k) \in \{0,1\}^k}
		\prod_{i=1}^{s-1} \Psi\left(\frac{x_{\tilde{\epsilon}_i}}{x_{\tilde{\epsilon}_{i+1}}}\right)
		\prod_{j=1}^k (G(x_j))^{n_j - n_{j-1} - \epsilon_j} \left(\frac{(n_j - n_{j-1})H(x_j)}{x_j^\alpha}\right)^{\epsilon_j},
\end{eqnarray*}
where
\[
		\prod_{i=1}^{s-1} \Psi\left(\frac{x_{\tilde{\epsilon}_i}}{x_{\tilde{\epsilon}_{i+1}}}\right) = 1 \ \ \ {\rm for} \ \ \ s \in \{0,1\}.
\]
\end{Theorem}
\noindent {\bf Proof.}
First, observe that
\begin{eqnarray*}
		\lefteqn{ \pr(X_{n_k} \le x_k, X_{n_{k-1}} \le x_{k-1}, \cdots, X_{n_1} \le x_1)}\\
        & = &
		\int\limits_{0}^{x_1}\int\limits_{0}^{x_2} \cdots \int\limits_{0}^{x_k} P_{n_k-n_{k-1}}(y_{k-1}, dy_k) P_{n_{k-1}-n_{k-2}}(y_{k-2}, dy_{k-1}) \cdots P_{n_1}(0, dy_1).
\end{eqnarray*}
Moreover, by the definition of $\Psi(\cdot)$, for any $a > 0$, we have
\[
    \lim_{x_{k+1} \to \infty} \Psi\left(\frac{a}{x_{k+1}}\right)
=
    \Psi\left(0\right)
=
    1
.
\]
Now in order to complete the proof it suffices to apply
Lemma \ref{thm:a} with $y_0=0$ and $x_{k+1} \to \infty$.
\qed\\

Finally, we present the cumulative distribution functions and characterizations of the finite-dimensional distributions
of the process $\{X_n\}$ for the most interesting examples of unit step distributions $\nu$.
Since, by Theorem \ref{thm:b}, finite-dimentional distributions of $\{X_n\}$ are uniquely determined by
the Williamson transform $G(\cdot)$
and the truncated $\alpha$-moment $H(\cdot)$ of the step distribution $\nu$, then this two characteristics are presented for each examples of the analyzed cases. Additionally in each example we derive the cdf of $\nu^{\vartriangle_{\alpha}n}$ that is the one-dimentional distribution of the process $\{X_n\}$.

We start with a basic case of a point-mass distribution $\nu$.

\begin{Example}\label{ex:3}
Let $\nu = \delta_1$.
Then the Williamson transform and truncated $\alpha$-moment of measure $\nu$ are given by
\[
G(x) = \left(1-\frac{1}{x^{\alpha}}\right)_+ \ \ \  \hbox{and} \ \ \ H(x)= \pmb{1}_{[1,\infty)}(x),
\]
respectively.
Hence, by Proposition \ref{prop:Fn} (ii), for any $n = 2,3,...$, we have
\[
F_n(x) = \left(1+\frac{n-1}{x^{\alpha}}\right) \left(1-\frac{1}{x^{\alpha}}\right)^{n-1} \pmb{1}_{[1,\infty)}(x).
\]

\end{Example}

In the next example we consider a linear combination of $\delta_1$ and the Pareto distribution that plays a crucial role in construction of Kenadall convolution.

\begin{Example}\label{ex:4}
Let $\nu = p \delta_1 + (1-p) \pi_{p}$, where $p\in(0,1]$ and $\pi_{p}$ is a Pareto distribution with the pdf \refs{pareto:pdf} with $2\alpha = p$.
Then Williamson transform and truncated $\alpha$-moment of measure $\nu$ are given by
\[
G(x) = \left\{ \begin{array}{l} \left(1-\frac{\alpha(1-p)}{(\alpha-p)}x^{-p} + \frac{p(1-\alpha)}{(\alpha-p)}x^{-\alpha}\right)\pmb{1}_{[1,\infty)}(x) \quad \hbox{if} \; p \neq \alpha,
\\[2mm]
\left(1 - (1-p) x^{-p} -  px^{-\alpha} +p (1-p)  x^{-\alpha} \log(x) \right)\pmb{1}_{[1,\infty)}(x) \quad \hbox{if} \; p = \alpha
 \end{array} \right.
\]
and
\[
H(x)= \left\{ \begin{array}{l}
 \left( \frac{p(1-\alpha)}{(p-\alpha)}+ \frac{p(1-p)}{(\alpha-p)} x^{\alpha-p} \right) \pmb{1}_{[1,\infty)}(x) \quad \hbox{if} \; p \neq \alpha,
\\[2mm]
p + p(1-p) \log(x) \quad \hbox{if} \; p = \alpha,
 \end{array} \right.
\]
respectively. Hence, by Proposition \ref{prop:Fn} (ii), for any $n = 2,3,...$, we have
\begin{eqnarray*}
F_n(x)
& = &  \left[1 -         \frac{\alpha(1-p)}{\alpha-p}x^{-p} + \frac{p(1-\alpha)}{\alpha-p}x^{-\alpha}\right]^{n-1}\\
&\cdot&
\left[1+\frac{(1-p)(np-\alpha)}{\alpha-p}x^{-p}-\frac{p(1-\alpha)(n-1)}{\alpha-p}x^{-\alpha}\right] \pmb{1}_{[1,\infty)}(x)
\end{eqnarray*}
for $p \neq \alpha$, and
\begin{eqnarray*}
F_n(x)
& = &  \left[1 - (1-p) x^{-p} -  px^{-\alpha} +p (1-p)  x^{-\alpha} \log(x) \right]^{n-1}\\
&\cdot&
\left[1 + (n-1) p x^{-\alpha} - (1-p) x^{-p} + p (1-p) (n+1)  x^{-\alpha} \log(x)\right] \pmb{1}_{[1,\infty)}(x)
\end{eqnarray*}
for $p = \alpha$.
\end{Example}
In the next example we consider the distribution $\nu$ with the lack of memory property for the Kendall convolution.
We refer to \cite{Poisson} for a general result about the existence of measures with the lack of memory property for the so called
monotonic generalized convolutions.

\begin{Example}\label{ex:5}
Let $\nu$ be a probability measure with the cdf $F(x) = 1 - (1-x^{\alpha})_+$, where $\alpha>0$.
Then the Williamson transform and truncated $\alpha$-moment of measure $\nu$ are given by
\[
G(x) = \frac{x^{\alpha}}{2}\pmb{1}_{[0,1)}(x) + \left(1-\frac{1}{2x^{\alpha}}\right) \pmb{1}_{[1,\infty)}(x)\ \ \ {\rm and} \ \ \
H(x) = \frac{x^{2\alpha}}{2}\pmb{1}_{[0,1)}(x) + \frac{1}{2} \pmb{1}_{[1,\infty)}(x),
\]
respectively.
Hence, by Proposition \ref{prop:Fn} (ii), for any $n = 2,3,...$, we have
\[
F_n(x) = \left\{ \begin{array}{lcl}
     \frac{1}{2}\frac{ n+1}{2^n}\, x^{\alpha n} & \hbox{ for } &  x \in [0,1]; \\[1mm]
     \frac{1}{2}\left(1 - \frac{1}{2x^{\alpha}} \right)^{n-1} \left( 1 + \frac{ n-1}{2x^{\alpha}} \right) & \hbox{ for } &  x > 1.
\end{array} \right.
\]

\end{Example}
In the next example we consider a unit step distribution, which is a stable probability measure for the Kendall random walk with unit step distribution $\nu$ (see Section \ref{sec:lim}, Theorem \ref{thm:CLT}).
\begin{Example}
Let $\rho_{\nu,\alpha}, \alpha > 0$ be a probability measure with cdf
\[
	F(x)
=
	\left( 1 + m_{\nu}^{(\alpha)} x^{-\alpha} \right) e^{-m_{\nu}^{(\alpha)} x^{-\alpha}} \pmb{1}_{(0,\infty)}(x),
\]
where $\alpha > 0$ and $m_{\nu}^{(\alpha)}>0$ is a parameter. Then the Williamson transform and truncated $\alpha$-moment of measure $\rho_{\nu,\alpha}$ are given by
\[
G(x) = \exp\{-m_{\nu}^{(\alpha)} x^{-\alpha}\} \pmb{1}_{(0,\infty)}(x) \ \ \
{\rm and} \ \ \
H(x) = m_{\nu}^{(\alpha)} G(x),
\]
respectively. Hence, by Proposition \ref{prop:Fn} (ii), for any $n = 2,3,...$, we have
\[
    F_n(x)
=
    \left( 1 + n m_{\nu}^{(\alpha)}  x^{-\alpha} \right) \exp\{- n m_{\nu}^{(\alpha)} x^{-\alpha}\} \pmb{1}_{(0,\infty)}(x).
\]
\end{Example}
\begin{Example}\label{ex:2}
Let $\nu=U(0,1)$ be the uniform distribution with the density $\nu(dx) = \pmb{1}_{(0,1)}(x) dx$.
Then the Williamson transform and truncated $\alpha$-moment of measure $\nu$ are given by
\[
G(x) = (x \wedge 1) - \frac{(x \wedge 1)^{\alpha+1}}{(\alpha+1)x^{\alpha}}, \ \ \
{\rm and} \ \ \
H(x) = \frac{(x \wedge 1)^{\alpha+1}}{(\alpha+1)},
\]
respectively. Hence, by Proposition \ref{prop:Fn} (ii), for any $n = 2,3,...$, we have
\begin{eqnarray*}
\lefteqn{
F_n(x) = \left(\frac{\alpha}{\alpha+1}\right)^n\left(1+\frac{n}{\alpha}\right)x^n \pmb{1}_{[0,1)}(x)}\\
& & + \left(1-\frac{1}{(\alpha+1)x^{\alpha}}\right)^{n-1}\left(1+\frac{n-1}{( \alpha+1)x^{\alpha}}\right)\pmb{1}_{[1,\infty)}(x).
\end{eqnarray*}

\end{Example}
\begin{Example}\label{ex:6}
Let $\nu = \gamma_{a,b},\; a,b>0$, be the Gamma distribution with the pdf
\[
\gamma_{a,b}(dx)=\frac{b^a}{\Gamma(a)} x^{a-1} e^{-bx} \pmb{1}_{(0,\infty)}(x) dx.
\]
Then the Williamson transform and truncated $\alpha$-moment of measure $\nu$ are given by
\[
G(x) = \gamma_{a,b}(0,x] - \frac{\Gamma(a+\alpha)}{\Gamma(a)} b^{-\alpha} x^{-\alpha}  \gamma_{a+\alpha,b}(0,x],\ \ \ {\rm and}\ \ \  H(x) = \frac{\Gamma(a+\alpha)}{\Gamma(a)} b^{-\alpha} \gamma_{a+\alpha,b}(0,x],
\]
where $\gamma_{a,b}(0,x] = \frac{b^a}{\Gamma(a)} \! \int_0^x t^{s-1} e^{-t} dt$. Hence, by Proposition \ref{prop:Fn} (ii),
for any $n = 2,3,...$, we have
\[
F_n(x) = \left[\gamma_{a,b}(0,x] - \frac{\Gamma(a+\alpha)}{\Gamma(a)} b^{-\alpha} x^{-\alpha}  \gamma_{a+\alpha,b}(0,x]\right]^{n-1} \cdot \left[ \gamma_{a,b}(0,x] + \frac{\Gamma(a+\alpha)}{\Gamma(a)} (n-1)  b^{-\alpha} x^{-\alpha} \gamma_{a+\alpha,b}(0,x] \right].
\]
\end{Example}
\section{Limit theorems} \label{sec:lim}
In this section we investigate limiting behaviors of Kendall random walks and connected continuous time processes.
The analysis is based on inverting the Williamson transform as the given in Proposition \ref{prop:Fn}, (ii). Moreover, as it is shown in Section \ref{sec:def}, the Kendall convolution is strongly related to the Pareto distribution. Hence, regular variation techniques play a crucial role in the analysis of the asymptotic behaviors and limit theorems for the processes studied in this section.

We start with the analysis of asymptotic behavior of the tail distribution of random variables $X_n$.

\begin{Theorem} \label{th.fn.as}
Let $\{X_n : n \in \mathbb{N}_0\}$ be a Kendall random walk with parameter $\alpha>0$ and unit step distribution $\nu \in\mathcal{P}_+$. Then

\[
		\overline{F}_n(x)
	=
		n\overline{F}(x) + \frac{1}{2} n (n-1) (H(x))^2 x^{-2\alpha} (1 + o(1))
\]
as $x \toi$.
\end{Theorem}
The proof of Theorem \ref{th.fn.as} is presented in Section \ref{th.fn.as:proof}.\\

The following Corollary is a direct consequence of the Theorem \ref{th.fn.as}.
\begin{Corollary}
Let $\{X_n : n \in \mathbb{N}_0\}$ be a Kendall random walk with parameter $\alpha>0$ and unit step distribution $\nu \in\mathcal{P}_+$. Moreover, let\\
{\rm (i)} $\overline{F}(x)$ be regularly varying with parameter
$\theta - \alpha$ as $x \toi$, where $0 < \theta < \alpha$. Then
\[
		\overline{F}_n(x)
	=
		n\overline{F}(x) (1 + o(1)) \ \ \ \ \ {\rm as} \ x \toi.
\]
{\rm (ii)} $m_\nu^{(\alpha)} < \infty$. Then
\[
		\overline{F}_n(x)
	=
		n\overline{F}(x) + \frac{1}{2} n (n-1) \left(m_\nu^{(\alpha)}\right)^2 x^{-2\alpha} (1 + o(1)) \ \ \ \ \ {\rm as} \ x \toi.
\]
{\rm (iii)} $\overline{F}(x) = o\left(x^{-2\alpha}\right)$ as $x \toi$. Then
\[
		\overline{F}_n(x)
	=
		\frac{1}{2} n (n-1) \left(m_\nu^{(\alpha)}\right)^2 x^{-2\alpha}(1 + o(1)) \ \ \ \ \ {\rm as} \ x \toi.
\]
\end{Corollary}

\begin{Remark}
It shows that in case of regularly varying step distribution $\nu$, the tail distribution of random variable $X_n$ is asymptotically equivalent to the maximum of $n$ i.i.d. random variables with distribution $\nu$. \\
\end{Remark}

In the next proposition, we investigate the limit distribution for Kendall random walks in case of finite $\alpha$-moment as well as for regularly varying tail of the unit step. We start with the following observation.

\begin{Remark}
Due to Proposition 1 in \cite{Bin71} random variable $X$ belongs to the domain of attraction of a stable measure with respect to
the Kendall convolution if and only if $1 - G(t)$ is regularly varying function at $\infty$.
\end{Remark}

Notice that $1 - G(t)$ is regularly varying whenever the random variable $X$ has finite $\alpha$-moment or its tail is regularly varying at infinity. The following Proposition formalizes this observation
providing formulas for stable distributions with respect to Kendall convolution.

\begin{Proposition}\label{thm:CLT} 
Let $\{X_n : n \in \mathbb{N}_0\}$ be a Kendall random walk with parameter $\alpha>0$ and unit step distribution $\nu \in\mathcal{P}_+$
	\\
{\rm (i)} If $m_{\nu}^{(\alpha)} <\infty$, then as $n \toi$,
\[
n^{-1/\alpha} X_n  \stackrel{d}{\to} X,
\]
where the cdf of random variable $X$ is given by
\begin{eqnarray} \label{cdf:mu}
		\rho_{\nu,\alpha,\theta}(0,x]
	=
			\left( 1 + m_{\nu}^{(\alpha)} x^{-\alpha} \right) e^{-m_{\nu}^{(\alpha)} x^{-\alpha}} \pmb{1}_{(0,\infty)}(x)
\end{eqnarray}
and the pdf of $X$ is given by
\begin{equation} \label{pdf:mu}
    \rho_{\nu,\alpha}(dx)
=
    \alpha  \left( m_{\nu}^{(\alpha)} \right)^2 x^{-2\alpha-1} \exp\{-m_{\nu}^{(\alpha)} x^{-\alpha}\} \pmb{1}_{(0,\infty)}(y) dx.
\end{equation}
{\rm (ii)} If $\overline{F}$ is regularly varying as $x \toi$ with parameter $\theta - \alpha$,
where $0 \leqslant \theta < \alpha$, then there exists a sequence $\{a_n\}$, $a_n \to \infty$, such that
\[
a_n^{-1} X_n  \stackrel{d}{\to}  X,
\]
where the cdf of random variable $X$ is given by
\begin{eqnarray} \label{cdf:reg}
		\rho_{\nu,\alpha,\theta}(0,x]
	=
			\left( 1 + x^{-(\alpha-\theta)} \right) e^{-x^{-(\alpha-\theta)}}\pmb{1}_{(0,\infty)}(x)
\end{eqnarray}
and the pdf of $X$ is given by
\begin{equation} \label{pdf:reg}
\rho_{\nu,\alpha,\theta}(dx)=\alpha  x^{-2(\alpha-\theta)-1} \exp\{-x^{-(\alpha-\theta)}\} \pmb{1}_{(0,\infty)}(x) dx.
\end{equation}

\end{Proposition}

The proof of Proposition \ref{thm:CLT} is presented in Section \ref{thm:CLT:proof}.\\

Now we define a new stochastic process $\{Z_n(t): n \in \mathbb{N}_0\}$ connected with the Kendall random walk $\{X_n: n \in \mathbb{N}_0\}$ such that
\[
\left\{Z_n(t)\right\} \stackrel{d}{=} \left\{a_n^{-1} X_{[nt]}\right\},
\]
where $[\cdot]$ denotes integer part and the sequence $\{a_n\}$ is such that $a_n > 0$ and $\lim_{n \to \infty} a_n = \infty$.

In the following theorem, we prove convergence of the finite-dimensional distributions of the process
$\{Z_n(t)\}$, for appropriately chosen sequence $\{a_n\}$.
\begin{Theorem}\label{thm:fidi_RV}
Let $\{X_n : n \in \mathbb{N}_0\}$ be a Kendall random walk with parameter $\alpha>0$ and unit step distribution $\nu \in\mathcal{P}_+$.
\begin{itemize}
\item[{\rm (i)}] If $m_{\nu}^{(\alpha)}<\infty $ and
$
a_n = n^{1/\alpha} (1 + o(1))$, as $n\to \infty$, then
\[
\{Z_n(t)\}\stackrel{fdd}{\rightarrow} \{Z(t)\},
\]
where, for any
$0 = t_0 \le t_1 \le ... \le t_k$, the finite-dimensional distributions of $\{Z(t)\}$ are given by
\begin{eqnarray*}
\lefteqn{
P \left( Z(t_1)\le z_1, Z(t_2)\le z_2,\cdots, Z(t_k)\le z_k \right)}\\
& = &
\sum\limits_{({\epsilon}_1,{\epsilon}_2,\cdots,{\epsilon}_k) \in \{0,1\}^k}
		\prod_{i=1}^{s-1} \Psi\left(\frac{z_{\tilde{\epsilon}_i}}{z_{\tilde{\epsilon}_{i+1}}}\right)
		\prod_{j=1}^k \left(\frac{\left(t_j - t_{j-1}\right)}{z_j^\alpha} m_{\nu}^{(\alpha)} \right)^{\epsilon_j} \exp\left\{ - m_{\nu}^{(\alpha)} \sum\limits_{i=1}^kz_i^{-\alpha} (t_i-t_{i-1})\right\},
\end{eqnarray*}
\item[{\rm (ii)}] If $\overline{F}(\cdot)$ is regularly varying as $x \toi$ with parameter $\theta - \alpha$, where $0 \leqslant \theta < \alpha$, then there exists a sequence $\{a_n\}, a_n \toi$ such that
\[
Z_n(t)\stackrel{fdd}{\rightarrow} Z(t),
\]
where, for any
$0 = t_0 \le t_1 \le ... \le t_k$, the finite-dimensional distributions of $\{Z(t)\}$ are given by
\begin{eqnarray*}
\lefteqn{
P \left( Z(t_1)\le z_1, Z(t_2)\le z_2,\cdots, Z(t_k)\le z_k \right)}\\
& = &
\sum\limits_{({\epsilon}_1,{\epsilon}_2,\cdots,{\epsilon}_k) \in \{0,1\}^k}
		\prod_{i=1}^{s-1} \Psi\left(\frac{z_{\tilde{\epsilon}_i}}{z_{\tilde{\epsilon}_{i+1}}}\right)
		\prod_{j=1}^k \left(\left(t_j - t_{j-1}\right) z_j^{\theta-\alpha} \right)^{\epsilon_j} \exp\left\{ -  \sum\limits_{i=1}^kz_i^{\theta - \alpha} (t_i-t_{i-1})\right\},
\end{eqnarray*}
\end{itemize}
where in both above cases we have
\[
		\prod_{i=1}^{s-1} \Psi\left(\frac{z_{\tilde{\epsilon}_i}}{z_{\tilde{\epsilon}_{i+1}}}\right) = 1 \ \ \ {\rm for} \ \ \ s \in \{0,1\}
\]
with
\[
\tilde{\epsilon}_1 = \min \left\{i : \epsilon_i = 1 \right\},  \; \dots ,
\tilde{\epsilon}_m := \min \left\{i>\tilde{\epsilon}_{m-1} : \epsilon_i = 1 \right\}, \ \ \ m=1,2, \dots , s, \ \ \ s = \sum_{i = 1}^k \epsilon_i.
\]

\end{Theorem}

The proof of Theorem \ref{thm:fidi_RV} is presented in Section \ref{thm:fidi_RV:proof}.\\

\section{Proofs}
In this section, we present detailed proofs of our results.

\subsection{Proof of Proposition \ref{kernel:prop}} \label{kernel:prop:proof}
Due to the independence of sequences $\{Y_k\}$, $\{\xi_k\}$, and $\{\theta_k\}$, it follows directly from the Definition \ref{def_krm} that the process $\{X_n\}$ satisfies the Markov property. \\
Now, let $n \in \mathbb{N}$ be fixed.
We shall show that, for all $k \in \mathbb{N}, A \in \mathcal{B}(\mathbb{R}_{+}), x\ge 0, \alpha >0$, the transition probabilities of the process $\{X_n\}$ are of the form \refs{kernel}.
In order to do this we proceed by induction. By Definition \ref{def_krm} we have
\begin{eqnarray*}
\lefteqn{
\pr \left( X_n \in A | X_{n-1} = x\right)}\\
&=&
\int\limits_{0}^{\infty} \pr \left( X_n \in A | X_{n-1} = x, Y_n = y \right) \nu(dy) \\
&=& \int\limits_{0}^{\infty}  \left\{ \pr \left( \max(x,y) \theta_n \in A, \xi_n <  \left(\frac{\min(x,y)}{\max(x,y)}\right)^{\alpha} \right) +
\mathbf{I}_{A}(\max(x,y))
\pr \left(\xi_n > \left(\frac{\min(x,y)}{\max(x,y)}\right)^\alpha\right)  \right\} \nu(dy).
\end{eqnarray*}
Moreover, by the independence of the random variables $\theta_n$ and $\xi_n$ the above expression is equal to
\begin{eqnarray}
\nonumber
\lefteqn{ \int\limits_{0}^{\infty} \left\{\left(\frac{\min(x,y)}{\max(x,y)}\right)^{\alpha} \pr \left( \max(x,y) \theta_n \in A \right) +  \left(1-\left(\frac{\min(x,y)}{\max(x,y)}\right)^{\alpha}\right) \mathbf{I}_{A}(\max(x,y))\right\}\nu(dy)} \\
&=& \nonumber
\int\limits_{0}^{\infty} T_{\max(x,y)}\left[\left(\frac{\min(x,y)}{\max(x,y)}\right)^{\alpha} \pi_{2\alpha}(A) +  \left(1-\left(\frac{\min(x,y)}{\max(x,y)}\right)^{\alpha}\right) \delta_1(A)\right]\nu(dy) \\
&=& \label{by:Kendal}
\int\limits_{0}^{\infty} \left(T_{\max(x,y)} \left(\delta_{\frac{\min(x,y)}{\max(x,y)}} \vartriangle_{\alpha} \delta_1\right)\right)(A)\nu(dy)\\
&=& \label{step_1}
\int\limits_{0}^{\infty} \left(\delta_x \vartriangle_{\alpha} \delta_y\right)(A)\nu(dy) =
\left(\delta_x \vartriangle_{\alpha} \nu \right)(A),
\end{eqnarray}
where \refs{by:Kendal} follows from Definition
\ref{Kendal:def} and \refs{step_1} follows from \refs{conv}.
This completes the first step of the proof by induction.

Now, assuming that
\begin{eqnarray} \label{step_k}
\pr \left( X_{n+k} \in A | X_n = x\right) = \left(\delta_x \vartriangle_{\alpha} \nu^{\vartriangle_{\alpha} k}\right) (A)
\end{eqnarray}
holds for $k \ge 2$ we establish its validity for $k + 1$.

Due to the Chapman-Kolmogorov equation for the process $\{X_n\}$ we have
\begin{eqnarray}
    \pr \left( X_{n+k+1} \in A | X_{n} = x\right)
&=& \nonumber
    \int\limits_{0}^{\infty}\int\limits_{A} P_1(y, dz) P_k(x, dy) \\
&=& \label{step_1k}
    \int\limits_{0}^{\infty}\left(\delta_y \vartriangle_{\alpha} \nu \right)(A) \left(\delta_x \vartriangle_{\alpha} \nu ^ {\vartriangle_{\alpha}  k} \right) (dy) \\
&=& \label{int_conv}
    \left(\delta_x \vartriangle_{\alpha} \nu ^ {\vartriangle_{\alpha}  k+1} \right)(A),
\end{eqnarray}
where \refs{step_1k} follows from  \refs{step_1} and \refs{step_k} while \refs{int_conv} follows from \refs{conv}.
This completes the induction argument and the proof.
\Halmos

\subsection{Proof of Lemma \ref{lem:transit}} \label{lem:transit:proof}
First notice that by Definition \ref{Kendal:def}, we have
\begin{eqnarray} \nonumber
	\left(\delta_x \vartriangle_{\alpha} \delta_y\right) ((0,t]) & = & \left(\frac{\min(x,y)}{\max(x,y)}\right)^{\alpha}
	\pr \left( \max(x,y) \theta \leq t \right) +  \left(1-\left(\frac{\min(x,y)}{\max(x,y)}\right)^{\alpha}\right)
	\mathbf{1}_{\{\max(x,y) < t\}}\\
& = &		\label{delta_conv}
	\left(1-\frac{x^{\alpha}y^{\alpha}}{t^{2\alpha}}\right) \mathbf{1}_{\{x < t, y < t\}} \\
&=&		\label{transit_deltas}
	\left[ \Psi\left(\frac{x}{t}\right) + \Psi\left(\frac{y}{t}\right) - \Psi\left(\frac{x}{t}\right)
	\Psi\left(\frac{y}{t}\right)\right] \mathbf{1}_{\{x < t, y < t\}},
\end{eqnarray}
where \refs{transit_deltas} is a direct application of \refs{Psi}.
In order to prove (i), observe that by \refs{conv} and \refs{transit_deltas} we have
\begin{eqnarray}
P_n (x, ((0,t]))
&=& \nonumber
\int\limits_0^{\infty} \left(\delta_x \vartriangle_{\alpha} \delta_y\right) (0,t) \nu^{\vartriangle_{\alpha} n}(dy)\\
& = & \nonumber
\int\limits_0^{t} \left[ \Psi\left(\frac{x}{t}\right) + \Psi\left(\frac{y}{t}\right) - \Psi\left(\frac{x}{t}\right) \Psi\left(\frac{y}{t}\right)\right] \mathbf{1}_{\{x < t \}}  \nu^{\vartriangle_{\alpha} n}(dy) \\
& = &  \label{psi_rep_G}
\left[\Psi \left( \frac{x}{t} \right) F_n(t) + \left(1-\Psi \left( \frac{x}{t} \right)\right) G_n(t)\right] \mathbf{1}_{\{x < t \}},
\end{eqnarray}
where \refs{psi_rep_G} holds by Definition \ref{Williamson:def}.
In order to complete the proof of the case (i) it suffices to combine \refs{psi_rep_G} with \refs{Williamson:n} and
Proposition \ref{prop:Fn} (ii).

To prove (ii), observe that integration by parts leads to
\begin{eqnarray}
 \int\limits_0^t w^{\alpha} \left(\delta_x \vartriangle_{\alpha} \delta_y\right) (dw)
&=& \nonumber
	t^{\alpha} \left(\delta_x \vartriangle_{\alpha} 	\delta_y\right)(0,t) - \int\limits_0^t \alpha w^{\alpha-1}
	\left(\delta_x \vartriangle_{\alpha} \delta_y\right)(0,w) dw \\
 &=& \label{trans_integ}
	\left( x^{\alpha} - 2 \frac{x^{\alpha} y^{\alpha}}{t^{\alpha}} + y^{\alpha} \right) \mathbf{1}_{\{x \vee y < t \}},
\end{eqnarray}
where \refs{trans_integ} follows from \refs{delta_conv}. Applying \refs{trans_integ} we obtain
\begin{eqnarray}
 \int\limits_0^t w^{\alpha} P_n(x,dw)
& = &  \nonumber
	\int\limits_0^{\infty} \int\limits_0^t w^{\alpha}
	\left(\delta_x \vartriangle_{\alpha} \delta_y \right) (dw) \nu^{\vartriangle_{\alpha} n}(dy) \\
&=& \nonumber
	\int\limits_0^t \left( x^{\alpha} - 2 \frac{x^{\alpha} y^{\alpha}}{t^{\alpha}} + y^{\alpha} \right)
	\nu^{\vartriangle_{\alpha} n}(dy) \mathbf{1}_{\{x < t \}} \\
& = & \label{Hn_integ}
	\left( x^{\alpha} F_n(t)- 2 \frac{x^{\alpha}}{t^{\alpha}} H_n(t) + H_n(t)\right) \mathbf{1}_{\{x < t \}},
\end{eqnarray}
where $H_n(t) := \int_0^t y^\alpha  \nu^{\vartriangle_{\alpha} n}(dy) = t ^{\alpha} \left( F_n(t) -  G_n(t) \right)$ by \refs{F:eq}.
Finally, the proof of the case (ii) is completed by combining  \refs{Hn_integ} with Proposition \ref{prop:Fn} (ii).
\Halmos

\subsection{Proof of Lemma \ref{thm:a}}	\label{thm:a:proof}
Let $k=1$. Then by Lemma \ref{lem:transit} we obtain
\begin{eqnarray*}
\lefteqn{
	\int\limits_{0}^{x_1}\Psi\left(\frac{y_1}{x_2}\right)  P_{n_1}(y_0, dy_1)
=
	P_{n_1}(y_0, ((0,x_1]))
	- x_2^{-\alpha} \int\limits_0^{x_1}  y_1^{\alpha} P_n(y_0, dy_1)}\\
	& = &
		\left[\Psi\left( \frac{y_0}{x_2} \right) G^{n_1}(x_1) +
		\frac{n_1}{x_1^{\alpha}} G^{n_1-1}(x_1) H_{1}(x_1) \Psi\left( \frac{y_0}{x_1} \right) \Psi\left( \frac{x_1}{x_2} \right)\right] \mathbf{1}_{\{y_0 < x_1\}},
\end{eqnarray*}
which ends the first step of proof by induction.

Now, assume that the formula holds for $k\in\mathbb{N}$. We shall establish its validity for $k+1$.
Let
\[
	\tilde{\eta}_1 := \min \left\{i \ge 2: \epsilon_i = 1 \right\},  \; \dots ,
	\tilde{\eta}_m := \min \left\{i>\tilde{\eta}_{m-1} : \epsilon_i = 1 \right\}, \ \ \ m=1,2, \dots , s_2, \ \ \ s_2 := \sum_{i = 2}^{k+1} \epsilon_i.
\]
Additionally, we denote
$s_1 := \sum_{i = 1}^{k+1} \epsilon_i$.\\
Moreover, let
\[
 \mathcal{A}_{k+1}^0 := \{ (0,\epsilon_2,\epsilon_3,\cdots,\epsilon_{k+1}) \in \{0,1\}^{k+1}: (\epsilon_2,\epsilon_3,\cdots,\epsilon_{k+1}) \in \mathcal{A}_k \},
\]
\[
 \mathcal{A}_{k+1}^1 := \{ (1,\epsilon_2,\epsilon_3,\cdots,\epsilon_{k+1}) \in \{0,1\}^{k+1}: (\epsilon_2,\epsilon_3,\cdots,\epsilon_{k+1}) \in \mathcal{A}_k \}.
\]
By splitting $\mathcal{A}_{k+1}$ into four subfamilies of sets: $\mathcal{A}_{k+1}^0$, $\mathcal{A}_{k+1}^1$, $\{(1,0,...,0)\}$, and $\{(0,...,0)\}$ and
applying the formula for $k$ and the first induction step we obtain
\begin{eqnarray}
\nonumber
\lefteqn{
\int\limits_{0}^{x_1}\left[\int\limits_{0}^{x_2} \cdots \int\limits_{0}^{x_{k+1}} \Psi\left(\frac{y_{k+1}}{x_{k+2}}\right) P_{n_{k+1}-n_{k}}(y_{k}, dy_{k+1}) \cdots P_{n_2-n_1}(y_1, dy_2) \right] P_{n_1}(y_0, dy_1)}\\
	\nonumber
	&=&
		\sum\limits_{(\epsilon_2,\epsilon_3,\cdots,\epsilon_{k+1}) \in \mathcal{A}_k}
		\Psi \left( \frac{x_{\tilde{\eta}_{s_2}}}{x_{k+2}} \right)
		\prod_{i=1}^{s_2-1} \Psi\left(\frac{x_{\tilde{\eta}_i}}{x_{\tilde{\eta}_{i+1}}}\right)
		\prod_{j=2}^{k+1} (G(x_j))^{n_j - n_{j-1} - \epsilon_j} \left(\frac{(n_j - n_{j-1})H(x_j)}{x_j^\alpha}\right)^{\epsilon_j}\\
	\nonumber
	&\cdot&	
		\int_0^{x_1}\Psi \left( \frac{y_1}{x_{\tilde{\eta}_1}} \right) P_{n_1}(y_0, dy_1)
		+ \prod_{j=2}^{k+1} (G(x_j))^{n_j - n_{j-1}} \int_0^{x_1}\Psi \left( \frac{y_1}{x_{k+2}} \right) P_{n_1}(y_0, dy_1)\\
	&=& \label{req_sum}
	    S[\mathcal{A}_{k+1}^0]+
	    S[\mathcal{A}_{k+1}^1]+
	    S[\{(1,0,...,0)\}]+
	    S[\{(0,...,0)\}],
\end{eqnarray}
where
\begin{eqnarray}
	 S[\mathcal{A}_{k+1}^0]
    &=&	 \nonumber
	 \sum\limits_{(0,\epsilon_2,\epsilon_3,\cdots,\epsilon_{k+1}) \in \mathcal{A}_{k+1}^0}
		\Psi \left( \frac{y_0}{x_{\tilde{\eta}_1}} \right)
		\Psi \left( \frac{x_{\tilde{\eta}_{s_2}}}{x_{k+2}} \right)
		\prod_{i=1}^{s_2-1} \Psi\left(\frac{x_{\tilde{\eta}_i}}{x_{\tilde{\eta}_{i+1}}}\right)\\
	\nonumber
	&\cdot&
		(G(x_1))^{n_1}\prod_{j=2}^{k+1} (G(x_j))^{n_j - n_{j-1} - \epsilon_j} \left(\frac{(n_j - n_{j-1})H(x_j)}{x_j^\alpha}\right)^{\epsilon_j},
\end{eqnarray}		
\begin{eqnarray}
        S[\mathcal{A}_{k+1}^1]
    &=& \nonumber
		\sum\limits_{(1,\epsilon_2,\epsilon_3,\cdots,\epsilon_{k+1}) \in \mathcal{A}_{k+1}^1}
		\Psi \left( \frac{y_0}{x_1} \right)
		\Psi \left( \frac{x_{\tilde{\eta}_{s_2}}}{x_{k+2}} \right)
		\prod_{i=1}^{s_2-1} \Psi\left(\frac{x_{\tilde{\eta}_i}}{x_{\tilde{\eta}_{i+1}}}\right)
		\Psi\left(\frac{x_1}{x_{\tilde{\eta}_{1}}}\right)\\
	\nonumber
	&\cdot&		
		\frac{n_1}{x_1^{\alpha}} (G(x_1))^{n_1-1} H_{1}(x_1)
		\prod_{j=2}^{k+1} (G(x_j))^{n_j - n_{j-1} - \epsilon_j} \left(\frac{(n_j - n_{j-1})H(x_j)}{x_j^\alpha}\right)^{\epsilon_j},
\end{eqnarray}
\begin{eqnarray*}
        S[\{(1,0,...,0)\}]
    =
        \frac{n_1}{x_1^{\alpha}} G^{n_1-1}(x_1) H_{1}(x_1) \Psi\left( \frac{y_0}{x_1} \right) \Psi\left( \frac{x_1}{x_{k+2}} \right) \prod_{j=2}^{k+1} (G(x_j))^{n_j - n_{j-1}},
\end{eqnarray*}
\begin{eqnarray*}
        S[\{(0,...,0)\}]
    =
        \Psi \left( \frac{y_0}{x_{k+2}} \right)  G(x_1)^{n_1}\prod_{j=2}^{k+1} (G(x_j))^{n_j - n_{j-1}}.
\end{eqnarray*}

Observe that for any sequence $(0,\epsilon_2,\epsilon_3,\cdots,\epsilon_{k+1}) \in \mathcal{A}_{k+1}^0$ we have
$(\tilde{\epsilon}_1, \tilde{\epsilon}_2,...,\tilde{\epsilon}_{s_1}) = (\tilde{\eta}_1, \tilde{\eta}_2, ..., \tilde{\eta}_{s_2}) $, with
$s_1 = s_2$, which implies that
\begin{eqnarray*}
		\Psi \left( \frac{y_0}{x_{\tilde{\eta}_1}} \right)
		\Psi \left( \frac{x_{\tilde{\eta}_{s_2}}}{x_{k+2}} \right)
		\prod_{i=1}^{s_2-1} \Psi\left(\frac{x_{\tilde{\eta}_i}}{x_{\tilde{\eta}_{i+1}}}\right)
	=
		\Psi \left( \frac{y_0}{x_{\tilde{\epsilon}_1}} \right)\Psi \left( \frac{x_{\tilde{\epsilon}_{s_1}}}{x_{k+2}} \right)
		\prod_{i=1}^{s_1-1} \Psi\left(\frac{x_{\tilde{\epsilon}_i}}{x_{\tilde{\epsilon}_{i+1}}}\right).
\end{eqnarray*}
Morever
\begin{eqnarray*}
	\nonumber
	\lefteqn{
		(G(x_1))^{n_1}\prod_{j=2}^{k+1} (G(x_j))^{n_j - n_{j-1} - \epsilon_j} \left(\frac{(n_j - n_{j-1})H(x_j)}{x_j^\alpha}\right)^{\epsilon_j} }\\
	&=&
		\prod_{j=1}^{k+1} (G(x_j))^{n_j - n_{j-1} - \epsilon_j} \left(\frac{(n_j - n_{j-1})H(x_j)}{x_j^\alpha}\right)^{\epsilon_j}.
\end{eqnarray*}
Hence
\begin{eqnarray}
        S[\mathcal{A}_{k+1}^0]
	=   \label{A0}
       \hspace{-10mm} \sum\limits_{(0,\epsilon_2,\epsilon_3,\cdots,\epsilon_{k+1}) \in \mathcal{A}_{k+1}^0} \hspace{-7mm}
        \Psi \left( \frac{y_0}{x_{\tilde{\epsilon}_1}} \right)\Psi \left( \frac{x_{\tilde{\epsilon}_{s_1}}}{x_{k+2}} \right)
		\prod_{i=1}^{s_1-1} \Psi\left(\frac{x_{\tilde{\epsilon}_i}}{x_{\tilde{\epsilon}_{i+1}}}\right)
    \prod_{j=1}^{k+1} (G(x_j))^{n_j - n_{j-1} - \epsilon_j} \left(\frac{(n_j - n_{j-1})H(x_j)}{x_j^\alpha}\right)^{\epsilon_j}.
\end{eqnarray}

Analogously, for any sequence $(1,\epsilon_2,\epsilon_3,\cdots,\epsilon_{k+1}) \in \mathcal{A}_{k+1}^1$ we have
$(\tilde{\epsilon}_1, \tilde{\epsilon}_2,...,\tilde{\epsilon}_{s_1}) = (1, \tilde{\eta}_1, \tilde{\eta}_2, ..., \tilde{\eta}_{s_2})$,
with $s_1 = s_2 + 1$
which implies that
\begin{eqnarray}
		\Psi \left( \frac{y_0}{x_1} \right)
		\Psi\left(\frac{x_1}{x_{\tilde{\eta}_{1}}}\right)
		\Psi \left( \frac{x_{\tilde{\eta}_{s_2}}}{x_{k+2}} \right)
		\prod_{i=1}^{s_2-1} \Psi\left(\frac{x_{\tilde{\eta}_i}}{x_{\tilde{\eta}_{i+1}}}\right)
	\nonumber
	&=&
		\Psi \left( \frac{y_0}{x_{\tilde{\epsilon}_1}} \right)
		\Psi\left(\frac{x_1}{x_{\tilde{\eta}_{1}}}\right) \Psi \left( \frac{x_{\tilde{\epsilon}_{s_1}}}{x_{k+2}} \right)
		\prod_{i=1}^{s_1-2} \Psi\left(\frac{x_{\tilde{\epsilon}_{i+1}}}{x_{\tilde{\epsilon}_{i+2}}}\right) \\
	&=& \label{req_prod1}
	    \Psi \left( \frac{y_0}{x_{\tilde{\epsilon}_1}} \right)
	     \Psi \left( \frac{x_{\tilde{\epsilon}_{s_1}}}{x_{k+2}} \right)
	    \prod_{i=1}^{s_1-1} \Psi\left(\frac{x_{\tilde{\epsilon}_i}}{x_{\tilde{\epsilon}_{i+1}}}\right),
\end{eqnarray}
where \refs{req_prod1} is the consequence of
\[
		\Psi\left(\frac{x_1}{x_{\tilde{\eta}_{1}}}\right) \prod_{i=1}^{s_1-2} \Psi\left(\frac{x_{\tilde{\epsilon}_{i+1}}}{x_{\tilde{\epsilon}_{i+2}}}\right)
	=
		\prod\limits_{i=1}^{s_1-1} \Psi\left(\frac{x_{\tilde{\epsilon}_{i}}}{x_{\tilde{\epsilon}_{i+1}}}\right).
\]
Moreover,
\begin{eqnarray}
	\nonumber
	\lefteqn{
		\frac{n_1}{x_1^{\alpha}} (G(x_1))^{n_1-1} H_{1}(x_1)
		\prod_{j=2}^{k+1} (G(x_j))^{n_j - n_{j-1} - \epsilon_j} \left(\frac{(n_j - n_{j-1})H(x_j)}{x_j^\alpha}\right)^{\epsilon_j} }\\
	&=& \nonumber
		\prod_{j=1}^{k+1} (G(x_j))^{n_j - n_{j-1} - \epsilon_j} \left(\frac{(n_j - n_{j-1})H(x_j)}{x_j^\alpha}\right)^{\epsilon_j}.
\end{eqnarray}
Hence
\begin{eqnarray}
	    S[\mathcal{A}_{k+1}^1]
	=   \label{A1}
	    \hspace{-10mm}\sum\limits_{(1,\epsilon_2,\epsilon_3,\cdots,\epsilon_{k+1}) \in \mathcal{A}_{k+1}^1} \hspace{-7mm}
	    \Psi \left( \frac{y_0}{x_{\tilde{\epsilon}_1}} \right)
	     \Psi \left( \frac{x_{\tilde{\epsilon}_{s_1}}}{x_{k+2}} \right)
	    \prod_{i=1}^{s_1-1} \Psi\left(\frac{x_{\tilde{\epsilon}_i}}{x_{\tilde{\epsilon}_{i+1}}}\right)
	    \prod_{j=1}^{k+1} (G(x_j))^{n_j - n_{j-1} - \epsilon_j} \left(\frac{(n_j - n_{j-1})H(x_j)}{x_j^\alpha}\right)^{\epsilon_j}.
\end{eqnarray}
Additionally, observe that
\begin{eqnarray}
        S[\{(1,0,...,0)\}]
    = \label{S1}
        \Psi\left( \frac{y_0}{x_{\tilde{\epsilon}_1}} \right)
        \Psi\left( \frac{x_{s_1}}{x_{k+2}} \right)
        \frac{n_1 H_{1}(x_1) }{x_1^{\alpha}}
         \prod_{j=1}^{k+1} (G(x_j))^{n_j - n_{j-1} - \epsilon_j}
\end{eqnarray}
and due to the fact that $n_0 = 0$, we have
\begin{eqnarray}
        S[\{(0,...,0)\}]
    = \label{S0}
        \Psi \left( \frac{y_0}{x_{k+2}} \right)  \prod_{j=1}^{k+1} (G(x_j))^{n_j - n_{j-1}}.
\end{eqnarray}

Finally, by combining \refs{A0}, \refs{A1}, \refs{S1}, and \refs{S0} with \refs{req_sum} we obtain that
\begin{eqnarray*}
\nonumber
\lefteqn{
\int\limits_{0}^{x_1}\left[\int\limits_{0}^{x_2} \cdots \int\limits_{0}^{x_{k+1}} \Psi\left(\frac{y_{k+1}}{x_{k+2}}\right) P_{n_{k+1}-n_{k}}(y_{k}, dy_{k+1}) \cdots P_{n_2-n_1}(y_1, dy_2) \right] P_{n_1}(y_0, dy_1)}\\
	\nonumber
	&=&
	\sum\limits_{({\epsilon}_1,{\epsilon}_2,\cdots,{\epsilon}_{k+1}) \in \mathcal{A}_{k+1}}
		\Psi \left( \frac{y_0}{x_{\tilde{\epsilon}_1}} \right)\Psi \left( \frac{x_{\tilde{\epsilon}_{s_1}}}{x_{k+2}} \right)
		\prod_{i=1}^{s_1-1} \Psi\left(\frac{x_{\tilde{\epsilon}_i}}{x_{\tilde{\epsilon}_{i+1}}}\right)
		\prod_{j=1}^{k+1} (G(x_j))^{n_j - n_{j-1} - \epsilon_j} \left(\frac{(n_j - n_{j-1})H(x_j)}{x_j^\alpha}\right)^{\epsilon_j} \\
  &+& \Psi \left( \frac{y_0}{x_{k+1}} \right) \prod_{j=1}^{k+1} (G(x_j))^{n_j - n_{j-1}}.
\end{eqnarray*}

This completes the induction argument and the proof.
\Halmos

\subsection{Proof of Theorem \ref{th.fn.as}} \label{th.fn.as:proof}
Due to Proposition \ref{prop:Fn}, for any $n \ge 2$, we have
\begin{eqnarray}
		F_n(x)
	\nonumber
	&=&
		\left(F(x) - \frac{1}{x^\alpha} H(x)\right)^{n-1}\left(F(x) + \frac{n-1}{x^\alpha}H(x)\right) \\
	\nonumber
	&=&
		(F(x))^n + n\sum_{k = 1}^{n-1} (-1)^{k-1} {n-1 \choose k-1} \frac{k-1}{k}
		\left(\frac{H(x)}{x^{\alpha}}\right)^k (F(x))^{n-k} \\
	&+& \label{three_elements}
		(-1)^{n-1} (n-1) \left(\frac{H(x)}{x^{\alpha}}\right)^n,
\end{eqnarray}
where \refs{three_elements} follows by the observation that, for any $a \ge 0$, $n \ge 2$, we have
\[
		(1 - a)^{n-1} (1 + a(n-1)) = 1 + n\sum_{k=1}^{n-1} (-1)^{k-1} {n-1 \choose k-1} \frac{k-1}{k} a^k + (-1)^{n-1} (n-1) a^{n}.
\]		
Thus
\begin{eqnarray*} 
		\overline{F}_n(x)
	=
		I_1 + I_2,
\end{eqnarray*}
where		
\[
		I_1
	=	
		1- (F(x))^n
	=
		n\overline{F}(x) (1 + o(1))
\]
as $x \toi$,
and
\begin{eqnarray*}	
		I_2
	=
		n\sum_{k = 1}^{n-1} (-1)^k {n-1 \choose k-1} \frac{k-1}{k}
		\left(\frac{H(x)}{x^{\alpha}}\right)^k (F(x))^{n-k} + (-1)^n (n-1) \left(\frac{H(x)}{x^{\alpha}}\right)^n.
\end{eqnarray*}

Note that $\lim_{x \to \infty} \frac{H(x)}{x^{\alpha}} = 0$ for any measure $\nu \in \mathcal{P}_{+}$ and hence
\[
		n\sum_{k = 3}^{n-1} (-1)^k {n-1 \choose k-1} \frac{k-1}{k}
		\left(\frac{H(x)}{x^{\alpha}}\right)^k (F(x))^{n-k} + (-1)^n (n-1) \left(\frac{H(x)}{x^{\alpha}}\right)^n
	=
		o\left(\frac{1}{2} n (n-1) \left(\frac{H(x)}{x^{\alpha}}\right)^2\right),
\]
as $x \toi$. This completes the proof.
\Halmos

\subsection{Proof of Proposition \ref{thm:CLT}}		\label{thm:CLT:proof}
The following lemma plays a crucial role in further analysis.
\begin{Lemma} \label{Han}
Let $H \in RV_{\theta}$ with $0<\theta<\alpha$. Then, there exists a sequence $\{a_n\}$ such that
\[
\frac{H(a_n)}{(a_n)^{\alpha}}
 = \frac{1}{n} (1 +o(1))
\]
as $n\to \infty$.
\end{Lemma}
\begin{proof}
First, observe that $ W(x) := x^{\alpha}/H(x)$ is regularly varying function with parameter $\alpha - \theta$ as $x \toi$. Then, due to  Theorem 1.5.12. in \cite{Bin87}, there exists an increasing function $V(x)$ such that
$W(V(x)) = x(1 + o(1)),$ as $x \toi$. Now, in order to complete the proof it suffices to take $a_n = V(n)$.
\end{proof}

\noindent {\bf Proof of Proposition \ref{thm:CLT}.} Using \refs{Williamson:n} and \refs{F:eq}, the Williamson transform of $a_n^{-1} X_n$ is given by
\begin{eqnarray} \label{Ganz}
    \left[G\left(\frac{a_n}{x}\right)\right]^n = \left( F\left(\frac{a_n}{x}\right) - \frac{x^{\alpha}}{a_n^{\alpha}} H \left(\frac{a_n}{x}\right)\right)^n.
\end{eqnarray}
In order to prove (i)  observe that under assumption of the finiteness of $m_\nu^{(\alpha)}$ we have
\[
\lim\limits_{n \to \infty} H \left(\frac{n^{1/\alpha}}{x}\right) = m_{\nu}^{(\alpha)}
\ \ \ \ \
{\rm and}
\ \ \ \ \
\lim_{n \toi} F\left(\frac{n^{1/\alpha}}{x}\right) = 1,
\]
which, by \refs{Ganz}, yields that
\[
\lim\limits_{n \to \infty} \left[G\left(\frac{n^{1/\alpha}}{x}\right) \right]^n = e^{-m_{\nu}^{(\alpha)}  x^{\alpha}}.
\]

Due to Proposition \ref{prop:Fn}, (i) there exists uniquely determined random variable $X$ with cdf \refs{cdf:mu} and pdf \refs{pdf:mu} such that $e^{-m_{\nu}^{(\alpha)}  x^{\alpha}} $ is its Williamson transform. This completes the proof of the case (i). \\

In order to prove (ii), notice that, due to Theorem 1.5.8 in \cite{Bin87}  $\overline{F} \in RV_{\theta-\alpha}$, implies that $H \in RV_{\theta}$.
Hence, for any $z > 0$, we have
\[
H\left(\frac{a_n}{x}\right) = x^{-\theta} H(a_n) (1 + o(1)),
\]

Moreover, by Lemma \ref{Han} we can choose a sequence $\{a_n\}$ such that
\[
\frac{H(a_n)}{a_n^{\alpha}} = \frac{1}{n} (1 + o(1))
\]

as $n \toi$. Thus,
\begin{eqnarray*}
    \lim_{n \toi} G\left[\left(\frac{a_n}{x}\right)^n\right]
=
    \lim\limits_{n \to \infty}\left( F\left(\frac{a_n}{x}\right) - \frac{x^{\alpha}}{a_n^{\alpha}} H \left(\frac{a_n}{x}\right)\right)^n
 =
    e^{-x^{\alpha- \theta}}.
\end{eqnarray*}
Due to Proposition \ref{prop:Fn}, (i) there exists random variable $X$ with cdf \refs{cdf:reg} and pdf \refs{pdf:reg} such that $e^{-x^{\alpha- \theta}}$ is its Williamson transform.
This completes the proof.
\Halmos

\subsection{Proof of Theorem \ref{thm:fidi_RV}} \label{thm:fidi_RV:proof}
{\bf Proof.} Let $0 =: t_0 \le t_1 \le t_2 \le \cdots \le t_k$, where $k \in \mathbb{N}$. By Theorem \ref{thm:b}, the distribution of \\
$\left( Z_{n}(t_1), Z_{n}(t_2),\cdots, Z_{n}(t_k) \right)$ is given by
\begin{eqnarray}
\nonumber
\lefteqn{
    P \left( Z_{n}(t_1)\le z_1, Z_{n}(t_2)\le z_2,\cdots, Z_{n}(t_k)\le z_k \right)}\\
\nonumber
& = &
    P \left( X_{[n t_1]} \le a_n z_1,  X_{[n t_2]} \le a_n z_2,\cdots, X_{[n t_k]} \le a_n z_k \right) \\
& = & \label{fdd_zn}
	\sum\limits_{({\epsilon}_1,{\epsilon}_2,\cdots,{\epsilon}_k) \in \{0,1\}^k}
	\prod_{i=1}^{s-1} \Psi\left(\frac{z_{\tilde{\epsilon}_i}}{z_{\tilde{\epsilon}_{i+1}}}\right)
	\prod_{j=1}^k (G(a_n z_j))^{[n t_j] - [n t_{j-1}] - \epsilon_j} \left(\frac{([n t_j] - [n t_{j-1}] ) H(a_n z_j)}{a_n ^{\alpha} z_j^\alpha}\right)^{\epsilon_j},
\end{eqnarray}
where
\[
		\prod_{i=1}^{s-1} \Psi\left(\frac{z_{\tilde{\epsilon}_i}}{z_{\tilde{\epsilon}_{i+1}}}\right) = 1 \ \ \ {\rm for} \ \ \ s \in \{0,1\}.
\]
In analogous way to the proof of Theorem \ref{thm:CLT} (i) we obtain
\begin{equation} \label{G_zn_conv:mu}
\lim\limits_{n \toi} G\left(a_n z_j \right)^{[nt_j]-[nt_{j-1}]-\epsilon_j} =
		\lim_{n \toi} \left(F(a_n z_j) - \frac{H(a_n z_j)}{(a_n z_j)^\alpha} \right)^{[nt_j]-[nt_{j-1}]-\epsilon_j} = \exp\left\{ - m_{\nu}^{(\alpha)} \left(t_j-t_{j-1}\right)   z_j^{\theta-\alpha} \right\}.
\end{equation}
and
\begin{equation} \label{H_zn_conv:mu}
\lim\limits_{n \to \infty} \left(\frac{([n t_j] - [n t_{j-1}] ) H(a_n z_j)}{a_n ^{\alpha} z_j^\alpha}\right)^{\epsilon_j} =
\left(t_j - t_{j-1}\right) z_j^{-\alpha} m_{\nu}^{(\alpha)}
\end{equation}

In order to complete the proof of (i) it suffices to pass with $n \toi$ in
\refs{fdd_zn} applying  \refs{G_zn_conv:mu} and \refs{H_zn_conv:mu}.

In order to prove (ii), notice that similarly to the proof of Theorem \ref{thm:CLT}, for any $t_j, z_j > 0$ and $1 \le j \le k$, we obtain
\begin{eqnarray} \label{G_zn_conv}
\lim\limits_{n \toi} G\left(a_n z_j \right)^{[nt_j]-[nt_{j-1}]-\epsilon_j} =
		\lim_{n \toi} \left(F(a_n z_j) - \frac{H(a_n z_j)}{(a_n z_j)^\alpha} \right)^{[nt_j]-[nt_{j-1}]-\epsilon_j} = \exp\left\{ -\left(t_j-t_{j-1}\right)   z_j^{\theta-\alpha} \right\}
\end{eqnarray}
and
\begin{eqnarray} \label{H_zn_conv}
\lim\limits_{n \to \infty} \left(\frac{([n t_j] - [n t_{j-1}] ) H(a_n z_j)}{a_n ^{\alpha} z_j^\alpha}\right)^{\epsilon_j} =
\left(t_j - t_{j-1}\right) z_j^{\theta-\alpha}.
\end{eqnarray}

In order to complete the proof it suffices to pass with $n \toi$ in
\refs{fdd_zn} applying  \refs{G_zn_conv} and \refs{H_zn_conv}.
\Halmos\\

{\bf Acknowledgements.} B. Jasiulis-Go\l dyn and E. Omey were supported by "First order Kendall maximal autoregressive processes and their applications", within the POWROTY/REINTEGRATION programme of the Foundation for Polish Science co-financed by the European Union under the European Regional Development Fund.

\end{document}